\documentclass[11pt,twoside, final]{amsart}
\copyrightinfo{0}{Iranian Mathematical Society}
\pagespan{1}{\pageref*{LastPage}}
\usepackage{etoolbox,lastpage}
\commby{}
\date{\scriptsize   Received: , Accepted: .}
\usepackage{amsmath,amsthm,amscd,amsfonts,amssymb,enumerate}
\usepackage{graphicx}		
\usepackage{color}
\usepackage[colorlinks]{hyperref}
 
\newtheorem{theorem}{Theorem}[section]
\newtheorem{proposition}[theorem]{Proposition}
\newtheorem{lemma}[theorem]{Lemma}
\newtheorem{corollary}[theorem]{Corollary}
\theoremstyle{definition}
\newtheorem{definition}[theorem]{Definition}
\newtheorem{example}[theorem]{Example}
\theoremstyle{remark}
\newtheorem{remark}[theorem]{Remark}
\numberwithin{equation}{section}

\newtheorem*{notation}{Notation}

\def \d {\partial}

\def \sol {\dashv}
\def \sag {\vdash}

\def \C {\mathbb{C}}
\def \cat {cat$^1$-}

\def \tr {\triangleright}

\input xypic
\input xy
\xyoption{v2}
\xyoption{all}
\xyoption{2cell}

\newcommand{\id}{\mathrm{id}}

\begin{document}
 
\title[Limits in MCI]{Limits in Modified Categories of Interest} 

\author[K. Emir]{Kad\.{I}r Em\.{I}r$^*$}
\address[Kadir Emir]{Department of Mathematics and Computer Science, Eski\c{s}ehir Osmangazi University, Turkey.}
\email{kadiremir86@gmail.com}
 
\author[S. \c{C}etin]{Sel\.{I}m \c{C}et\.{I}n}
\address[Selim \c{C}etin]{Department of Mathematics, Mehmet Akif Ersoy University, Burdur, Turkey.}
\email{selimcetin@mehmetakif.edu.tr}

  \thanks{$^*$Corresponding author}

 \maketitle

\begin{abstract}
We firstly prove the completeness of the category of crossed modules in a modified category of interest. Afterwards, we define pullback crossed modules and pullback \cat objects that are both obtained by pullback diagrams with extra structures on certain arrows. These constructions unify many corresponding results for the cases of groups, commutative algebras and can also be adapted to various algebraic structures. \\
\textbf{Keywords:}  Modified category of interest, crossed module, \cat object, limit.  \\
\textbf{MSC(2010):}  18D05, 17A30, 18A30, 18A35.
\end{abstract}
 
\section{\bf Introduction}

The notion of category of interest was introduced to unify various properties of algebraic structures. The main idea is due to Higgins \cite{Hig} and the definition is improved by Orzech \cite{Orz}. As indicated in \cite{ CDL4, CDL5, Loday1, Lo3, LoRo, Orz}, many algebraic categories are the essential examples of category of interest. However the categories of cat$^{1}$-objects of Lie (associative, Leibniz, etc.) algebras are not. Because of this issue, the authors of \cite{BCDU} introduced a new type of this notion, called {\it modified category of interest} that satisfies all axioms of the former notion except one, which is replaced by a new and modified axiom. The main examples are those, which are equivalent to the categories of crossed modules in the categories of groups, (commutative) algebras, dialgebras, Lie and Leibniz algebras, etc. See \cite{bau,CMDA,dede,lue,Por} for more examples.

\medskip

Crossed modules were introduced by Whitehead in \cite{W1} as a model of homotopy 2-types and used to classify higher dimensional cohomology groups. The notion of crossed module is also defined for various algebraic structures. However the definition of crossed modules in modified
categories of interest unifies all of these definitions.  As an equivalent model of 
homotopy 2-types, cat$^{1}$-groups are introduced by Loday in \cite{Lswf}. This notion and the corresponding equivalence is also adapted to many algebraic structures, as well as to modified category of interest \cite{ESA}.

\medskip

In this paper, we firstly prove that the category of crossed modules in a modified category of interest $\C$ is finitely complete. This unifies a number of constructions given in \cite{Nizar}. Afterwards, we define pullback crossed modules and pullback \cat objects in $\C$ that are both obtained by pullback diagrams with extra categorical structures on certain arrows. These definitions will unify the constructions and results given in \cite{Alp1,Alp2,Brown}. Moreover, one can adapt them to many different algebraic structures such as Lie algebras, Leibniz algebras, dialgebras, etc.

\section{\bf Preliminaries}

In this section, we recall some notions from \cite{BCDU, Lswf, ESA}.

\subsection{Modified Category of Interest}

\begin{definition}
	Let $\mathbb{C}$ be a category of groups with a set of operations $\Omega$ and
	with a set of identities $\mathbb{E}$, such that $\mathbb{E}$ includes the
	group identities and the following conditions hold. If $\Omega_{i}$ is the set
	of $i$-ary operations in $\Omega$, then:
	
	\begin{enumerate}
		\item[(a)] $\Omega=\Omega_{0}\cup\Omega_{1}\cup\Omega_{2}$;
		
		\item[(b)] the group operations (written additively : $0,-,+$) are elements of
		$\Omega_{0}$, $\Omega_{1}$ and $\Omega_{2}$ respectively. Let $\Omega
		_{2}^{\prime}=\Omega_{2}\setminus\{+\}$, $\Omega_{1}^{\prime}=\Omega
		_{1}\setminus\{-\}.$ Assume that if $\ast\in\Omega_{2}$, then $\Omega
		_{2}^{\prime}$ contains $\ast^{\circ}$ defined by $x\ast^{\circ}y=y\ast x$ and
		assume $\Omega_{0}=\{0\}$;
		
		\item[(c)] for each $\ast\in\Omega_{2}^{\prime}$, $\mathbb{E}$ includes the
		identity $x\ast(y+z)=x\ast y+x\ast z$;
		
		\item[(d)] for each $\omega\in\Omega_{1}^{\prime}$ and $\ast\in\Omega
		_{2}^{\prime}$, $\mathbb{E}$ includes the identities $\omega(x+y)=\omega
		(x)+\omega(y)$ and \emph{either} the identity $\omega(x\ast y)=\omega(x)\ast \omega(y)$ \emph{or} the identity $\omega(x\ast y)=\omega(x)\ast y$.
	\end{enumerate}
	
	Denote by $\Omega'_{1S}$ the subset of those elements in $\Omega'_1$, which satisfy the identity $\omega(x \ast y) = \omega(x) \ast y $, and by $\Omega''_1$ all other unary operations, i.e. those which satisfy the first identity from (d).
	
	\medskip

	Let $C$ be an object of $\mathbb{C}$ and $x_{1},x_{2},x_{3}\in C$:
	
	\begin{enumerate}
		\item[(e)] $x_{1}+(x_{2}\ast x_{3})=(x_{2}\ast x_{3})+x_{1}$,
	for each $\ast\in\Omega_{2}^{\prime}$.\newline
	
	\item[(f)] For each ordered pair $(\ast,\overline{\ast}%
	)\in\Omega_{2}^{\prime}\times\Omega_{2}^{\prime}$ there is a word $W$ such
	that:
	\begin{eqnarray*}
	(x_{1}\ast x_{2}) \ \overline{\ast} \ x_{3}=W \big( x_{1}(x_{2}x_{3}),x_{1}(x_{3}
	x_{2}),(x_{2}x_{3})x_{1},
	(x_{3}x_{2})x_{1}, \\ x_{2}(x_{1}x_{3}),x_{2}(x_{3}x_{1}),(x_{1}x_{3})x_{2}
	,(x_{3}x_{1})x_{2} \big),
	\end{eqnarray*}
  \end{enumerate}
  where each juxtaposition represents an operation in $\Omega_{2}^{\prime}$.
  
  \medskip

	A category of groups with operations $\mathbb{C}$ satisfying
	conditions (a)-(f) is called a \textbf{modified
		category of interest}, or \textbf{MCI} for short.
\end{definition}

As indicated in \cite{BCDU}, the difference between this definition and that
of the original \textbf{category of interest} is the modification of the second identity in (d). According to this definition every category of interest is also a modified category of interest.

\begin{definition}
	Let $A,B$ be two objects of $\C$. A map $f \colon A \to B$ is called a morphism of $\C$ if it satisfies:
	\begin{align*}
	f(a + a') & = f (a) + f(a') , \\
	f(a \ast a') & = f(a) \ast f(a') ,
	\end{align*}
	for all $a,a' \in A$, $\ast\in\Omega_{2}^{\prime}$ and also commutes with all $w \in \Omega_{1}^{\prime}$.
\end{definition}

\begin{example}\label{example1}
	The categories of groups, algebras, commutative algebras, Lie algebras, Leibniz algebras, dialgebras are all (modified) categories of interest.
\end{example}

\begin{example}
	The categories \ $\mathbf{Cat}^{\mathbf{1}}\mathbf{Ass}$,  $\mathbf{Cat}
	^{\mathbf{1}}\mathbf{Lie}$, $\mathbf{Cat}^{\mathbf{1}}\mathbf{Leibniz}$, i.e.
	the categories of \cat associative algebras, \cat Lie algebras and \cat Leibniz algebras are the examples of modified categories of interest, which are not
	categories of interest (see \cite{BCDU} for details).
\end{example}

\begin{notation}
	From now on, $\mathbb{C}$ will denote an arbitrary but fixed modified category of interest.
\end{notation}

\begin{definition}
Let $B\in\mathbb{C}$. A subobject of $B$ is called an ideal if it is the
kernel of some morphism. 

\medskip

In other words, $A$ is an
ideal of $B$ if and only if  $A$ is a normal subgroup of $B$ and $a\ast b\in A,$ for all $a\in A$, $b\in B$ and $\ast \in\Omega_{2}^{\prime}$.
\end{definition}

\begin{definition}
	Let $A,B\in\mathbb{C}$. An extension of $B$ by $A$ is a sequence:
	\begin{align}\label{split}
	\xymatrix{0\ar[r]&A\ar[r]^-{i}&E\ar[r]^-{p}&B\ar[r]&0}
	\end{align}
	where $p$ is surjective and $i$ is the kernel of $p$. We say that an extension is split if there exists a morphism $s \colon B \to E$ such that $p s = 1_B$.
\end{definition}

\begin{definition}
	The split extension \eqref{split} induces an action of $B$ on $A$ corresponding to the operations of $\C$ with:
	\begin{align*}
	b\cdot a&=s(b)+a-s(b), \\
	b\ast a&=s(b)\ast a,
	\end{align*}
	for all $b\in B$, $a\in A$ and $\ast\in\Omega'_{2}.$ 
	
	\medskip
	
	Actions defined
	by the previous equations are called \textit{derived actions} of $B$ on $A$. Remark that we use the notation $`` \ast \text{''}$ to denote both the star operation and the star action.
	
	\medskip
	
	Given an action of $B$ on $A,$ a semi-direct product $A\rtimes B$ is a
	universal algebra, whose underlying set is $A\times B$ and the operations are
	defined by:
	\begin{align*}
	\omega(a,b) & = (\omega\left(  a\right)  ,\omega\left(  b\right)  ),\\
	(a^{\prime},b^{\prime})+(a,b) & = (a^{\prime}+b^{\prime}\cdot a,b^{\prime
	}+b),\\
	(a^{\prime},b^{\prime})\ast(a,b) & = (a^{\prime}\ast a+a^{\prime}\ast
	b+b^{\prime}\ast a,b^{\prime}\ast b),
	\end{align*}
	for all $a,a^{\prime}\in A,$ $b,b^{\prime}\in B$, $\ast\in\Omega_{2}^{\prime}$. An action of $B$ on $A$ is a derived action if and only if
	$A\rtimes B$ is an object of $\mathbb{C}$.
	
	\medskip
	
	Denote a general category of groups with operations of a modified category of interest $\C$ by $\C_G$. 	A set of actions of $B$ on $A$ in $\mathbb{C}_{G}$ is a set of
	derived actions if and only if it satisfies the following conditions:
	
	\begin{enumerate}

		\item[\textit{1.}] $0\cdot a=a$,
		
		\item[\textit{2.}] $b\cdot(a_{1}+a_{2})=b\cdot a_{1}+b\cdot a_{2}$,
		
		\item[\textit{3.}] $(b_{1}+b_{2})\cdot a=b_{1}\cdot(b_{2}\cdot a)$,
		
		\item[\textit{4.}] $b\ast(a_{1}+a_{2})=b\ast a_{1}+b\ast a_{2}$,
		
		\item[\textit{5.}] $(b_{1}+b_{2})\ast a=b_{1}\ast a+b_{2}\ast a$,
		
		\item[\textit{6.}] $(b_{1}\ast b_{2})\cdot(a_{1}\ast a_{2})=a_{1}\ast a_{2}$,
		
		\item[\textit{7.}] $(b_{1}\ast b_{2})\cdot(a\ast b)=a\ast b$,
		
		\item[\textit{8.}] $a_{1}\ast(b\cdot a_{2})=a_{1}\ast a_{2}$,
		
		\item[\textit{9.}] $b\ast(b_{1}\cdot a)=b\ast a$,
		
		\item[\textit{10.}] $\omega(b\cdot a)=\omega(b)\cdot\omega(a)$,
		
		\item[\textit{11.}] $\omega(a\ast b)=\omega(a)\ast b = a \ast \omega(b)$ for any $\omega \in \Omega'_{1S}$, and $\omega(a \ast b)=\omega(a) \ast \omega (b)$ for any $\omega \in \Omega''_{1}$,
		
		\item[\textit{12.}] $x\ast y+z\ast t=z\ast t+x\ast y$,
	\end{enumerate}
	for each $\omega\in\Omega_{1}^{\prime}$, $\ast\in\Omega_{2}^{\prime}$, $b$,
	$b_{1}$, $b_{2}\in B$, $a,a_{1},a_{2}\in A$; and for  $x,y,z,t\in A\cup B$
	whenever both sides of the last condition are defined.
\end{definition}

\subsection{Crossed Modules}

\begin{definition}
	A crossed module $(C_{1},C_{0},\partial)$ in $\mathbb{C}$ is given by a morphism $\partial \colon C_{1} \to C_{0}$ with a derived action of $C_{0}$ on $C_{1}$ such that:
	\medskip
	\begin{itemize}
		\item[XM1)]
		$\begin{array}{l}
		\partial(c_{0}\cdot c_{1})=c_{0}+\partial(c_{1})-c_{0} \\
		\partial(c_{0}\ast c_{1})=c_{0}\ast\partial(c_{1})
		\end{array}$
		\bigskip
		\item[XM2)]
		$\begin{array}{l}
		\partial(c_{1})\cdot c_{1}^{\prime}=c_{1} +c_{1}^{\prime}-c_{1} \\
		\partial(c_{1})\ast c_{1}^{\prime}=c_{1}\ast c_{1}^{\prime}
		\end{array}$
	\end{itemize}
	\medskip
	for all $c_{0}\in C_{0}$, $c_{1},c_{1}^{\prime}\in C_{1}$, $\ast\in\Omega_{2}^{\prime}$.
	
	\medskip
	
	A morphism between two crossed modules $(C_{1},C_{0},\partial
	)\to(C_{1}^{\prime},C_{0}^{\prime},\partial^{\prime})$ is a pair $(\mu_{1},\mu_{0})$ of morphisms
	$\mu_{0} \colon C_{0}\to C_{0}^{\prime}$, $\mu_{1} \colon C_{1}\to C_{1}^{\prime}$, such that the diagram:
	$$ \xymatrix @R=40pt@C=40pt{
		C_1 \ar[r]^{\d} \ar[d]_{\mu_1} &  C_0  \ar[d]^{\mu_0}  \\
		C'_1 \ar[r]_{\d'} & C'_0   } $$
	commutes and:
	\begin{align*}
	\mu_{1}(c_{0}\cdot c_{1})&=\mu_{0}(c_{0})\cdot\mu_{1}(c_{1}) \, ,\\
	\mu_{1}(c_{0}\ast c_{1})&=\mu_{0}(c_{0})\ast\mu_{1}%
	(c_{1}) \, ,
	\end{align*}
	for all $c_{0}\in C_{0}$, $c_{1}\in C_{1}$ and $\ast\in\Omega_{2}^{\prime}$.
\end{definition}

\medskip

Crossed modules and their morphisms form the category of crossed modules in $\C$ that will be denoted by $\mathbf{XMod}$.

\begin{example}\cite{JFM2}
	A crossed module of groups is given by a group homomorphism $\d\colon E \to G$, together with an action $\tr$ of $G$ on $E$ such that (for all $e,f \in E$ and $g \in G$):
	\begin{itemize}
		\item $\d(g \tr e) =g \,\d(e) \, g^{-1}$,
		\item $\d(e)  \tr f=e\, f \, e^{-1}$.
	\end{itemize}
\end{example}

\begin{example}\cite{JFM2}
	A crossed module of Lie algebras is given by a Lie algebra homomorphism $\d\colon \mathfrak{e} \to \mathfrak{g}$, together with an action $\tr$ of $\mathfrak{g}$ on $\mathfrak{e}$ such that (for all $e,f \in \mathfrak{e}$ and $g \in \mathfrak{g}$):
	\begin{itemize}
		\item $\d(g \tr e) = [g ,\d(e)]$,
		\item $\d(e)  \tr f= [e, f]$.
	\end{itemize}
\end{example}

Note that $\tr$ denotes the group action and the Lie algebra action respectively in the previous examples.

\subsection{{Cat}$^{1}$ Objects}

\begin{definition}
	Let $S$ be a subobject of $R$. A cat$^{1}$-object $(e;s,t,R \to S)$ in $\mathbb{C}$ is an object $C$ together with the morphisms $
	s,t \colon R \to S$ and $e \colon S \to R$ such that satisfying the following conditions:	
	\begin{itemize}
		\item $ s  e = \id_S$ and $t e = \id_S$,
		\item $x\ast y=0$, $x+y-x-y=0$,
	\end{itemize}
	for all $\ast\in\Omega_{2}^{\prime}$
	and $x \in \ker s$, $y \in \ker t$. 
	
	\bigskip
	
	Let $C=\left( e;s,t:R\rightarrow S\right) $ and $C^{\prime }=\left(
	e^{\prime };s^{\prime },t^{\prime }:R^{\prime }\rightarrow S^{\prime
	}\right) $ be two \cat objects. A \cat morphism $\left( \phi ,\varphi \right) \colon C \to C'$  is a tuple which consists of morphisms $\phi :R\rightarrow R^{\prime }$ and $\varphi
	:S\rightarrow S^{\prime }$ such that the following diagram commutes:	
	$$
	\diagram
	R\ddto\ddto<.5ex>^{s}_{t} %
	\rrto^{\phi} &&  R'\ddto \ddto<.5ex>^{s'}_{t'}\\
	\\
	S \rrto_{\varphi} \uutol^{{e}} && S' \uutor_{e'}
	\enddiagram
	$$
\end{definition}

Cat$^1$-objects and their morphisms form the category of \cat objects in $\C$ that will be denoted by $\mathbf{Cat^1}$.

\begin{notation}
	We denote any \cat object in $\C$ by $(R,S)$ for short.
\end{notation}

\begin{example}\cite{jas}
	A cat$^{1}$-Leibniz
	algebra consists of a Leibniz algebra $L$, a sub Leibniz algebra $M$ and Leibniz
	algebra homomorphisms: $ s, t \colon L \to M$ and $e \colon M \to L$ such that:
	\begin{itemize}
		\item $ s  e = \id_M$ and $t e = \id_M$,
		
		\item $[x, y]=0=[y, x]$,
	\end{itemize}
	for all $x \in \ker s$, $y \in \ker t$.
\end{example}

\begin{example}
	A cat$^1$-dialgebra consists of a dialgebra \cite{Loday1} $D$, a sub dialgebra and dialgebra homomorphisms: $ s,t \colon D \to F$ and $e \colon F \to D$  such that:
	\begin{itemize}
		\item $ se=\id_F$ and $te=\id_F$,
		
		\item  $x\sol  y=0=y \sol x$,
		$x \sag y=0=y \sag x $,
	\end{itemize}
	for all $x \in \ker s$, $y \in \ker t$.
\end{example}

\begin{proposition}
	The categories $\mathbf{XMod}$ and $\mathbf{Cat}^\mathbf{1} $ are equivalent.
\end{proposition}

\begin{proof}
	Let $(C_{1},C_{0},\partial )$ be a crossed module in $\mathbb{C}$. Consider
	the corresponding semi-direct product $C_{1}\rtimes C_{0}$ induced from the
	action of $C_{0}$ on $C_{1}$. By using the morphisms $s,t \colon C_{1}\rtimes
	C_{0}\to C_{0}$ and $e \colon C_0 \to C_{1}\rtimes
	C_{0}$ defined by $s(c_{1},c_{0})=c_{0}$, $t(c_{1},c_{0})=\d (c_{1})+c_{0}$ and $e(c_0)=(0,c_0)$, we obtain a \cat object. This yields to the functor $\mathbf{C^1} \colon \mathbf{XMod} \to \mathbf{Cat}^{\mathbf{1}}$. See \cite{ESA} for converse.
\end{proof}

\section{\bf Limits in MCI}

The cartesian product $P \times R$ is the product object of $P$ and $R$ in $\C$, with the projection morphisms satisfying the universal property. 

\medskip

Suppose that $\alpha :P\rightarrow S$ and $\beta :R\rightarrow S$ are two morphisms in $\C$. Then the subobject of the cartesian product: 
\begin{align*}
P\times _{S}R=\left\{ \left( p,r\right) \mid \alpha \left( p\right) =\beta
\left( r\right) \right\},
\end{align*} 
the {\it fiber product}, defines the pullback of $\alpha, \beta$. 

\medskip

Therefore a modified category of interest $\C$ has products and pullbacks which guarantees the existence of equalizer objects. Briefly, suppose that we have two parallel morphisms $f,g \colon P \to R $. Their equalizer is defined as $\mathrm{Eq}(f,g)=\{ x \in P \mid f(x)=g(x) \}$.

\medskip

Consequently, we can say that $\C$ has all finite limits since it has both products and equalizers. Thus $\C$ is finitely complete.

\subsection{Limits in Category of Crossed Modules in MCI}

\begin{definition}
	The category of crossed modules in $\C$ with fixed codomain $X$ forms a full subcategory of $\mathbf{XMod}$ that is denoted by $\mathbf{XMod/X}$. These kind of crossed modules will be called {\it crossed $X$-modules}. 
\end{definition}

\begin{lemma}\label{main1}
	Given two crossed modules $(P,S, \alpha)$ and $(R,S,\beta)$ there is a crossed module: 
	\begin{align*}
	\partial :P\times _{S}R\rightarrow S \, ,
	\end{align*}
	where $\partial \left(
	p,r\right) =\alpha \left( p\right) =\beta \left( r\right) $ and the action of $S$ on $P\times _{S}R$ is defined by:
	\begin{equation*}
    s \cdot \left( p,r\right) 
	=\left( s \cdot p, s \cdot r \right) ,
	\quad \quad \quad
    s \ast \left( p,r\right) 
	=\left(s \ast p,s\ast r\right) .
	\end{equation*}
\end{lemma}

\begin{proof}
	 The action given above is well-defined and the action conditions are already satisfied. Moreover $\partial :P\times _{S}R\rightarrow S$ is a morphism of $\C$ since:
	\begin{align*}
	\partial \left( \left( p,r\right) + \left( p^{\prime
	},r^{\prime }\right) \right) & =  \partial \left( p +
	p^{\prime },r + r^{\prime }\right)   \\
	& =  \alpha \left( p + p^{\prime }\right)   \\
	& =  \alpha \left( p\right) + \alpha \left( p^{\prime
	}\right)   \\
	& =  \partial \left( p,r\right) + \partial \left( p^{\prime
	},r^{\prime }\right) . 
	\end{align*}
	
	Similarly we have: 
	\begin{align*}
	\partial \left( \left( p,r\right) \ast \left( p^{\prime
	},r^{\prime }\right) \right) = \partial \left( p,r\right) \ast \partial \left( p^{\prime},r^{\prime }\right) ,
	\end{align*}
	for all $\left(
	p,r\right) , (p',r') \in P\times _{S}R$. Also $\d$ commutes with all $w \in \Omega_{1}^{\prime}$ since:
	\begin{align*}
	\d \big( w(p,r) \big) & = \d \big( w(p) , w(r) \big) = \alpha \big( w(p) \big) = w \big( \alpha (p) \big) = w \big( \d (p,r) \big).
	\end{align*}
	
	Finally, $\partial$ satisfies the crossed module conditions:
	\medskip
	\begin{itemize}
		\item[XM1)]
		\begin{align*}
		\partial \left( s \cdot \left( p,r\right)  \right) & = \partial \left(s \cdot
		p,s \cdot r\right)  = \alpha ( s \cdot p)  =  s + \alpha \left( p\right) - s = s + \partial \left( p,r\right) - s , \\
		\partial \left(s \ast \left( p,r\right) \right) & = \partial \left(
		s \ast p, s \ast r \right) 
		=  \alpha \left(s \ast p\right) 
		=   s \ast \alpha \left( p\right) 
		=  s \ast \partial \left( p,r\right) ,
		\end{align*}
		
		\item[XM2)] 
		\begin{align*}
		\partial \left( p^{\prime },r^{\prime }\right) \cdot \left( p,r\right)   & = 
		\alpha \left( p^{\prime }\right) \cdot \left( p,r\right)     \\
		& =  \left(\alpha \left( p^{\prime }\right) \cdot p , \alpha \left( p^{\prime }\right) \cdot r \right)   \\
		& =  \left(\alpha \left( p^{\prime }\right) \cdot p , \beta \left( r^{\prime }\right) \cdot r \right)   \\
		& =  \left( p' +  p - p', r' + r - r' \right)   \\
		& =  \left( p^{\prime },r^{\prime
		}\right) + \left( p,r\right) - (p',r') , 
		\end{align*}
		\begin{align*}
		\partial \left( p^{\prime },r^{\prime }\right) \ast \left( p,r\right)   & = 
		\alpha \left( p^{\prime }\right) \ast \left( p,r\right)     \\
		& =  \left(\alpha \left( p^{\prime }\right) \ast p , \alpha \left( p^{\prime }\right) \ast r \right)   \\
		& =  \left(\alpha \left( p^{\prime }\right) \ast p , \beta \left( r^{\prime }\right) \ast r \right)   \\
		& =  \left( p' \ast  p,r' \ast r \right)   \\
		& =  \left( p^{\prime },r^{\prime
		}\right) \ast \left( p,r\right) , 
		\end{align*}
	\end{itemize}
	for all $\left(p,r\right) , (p',r') \in P\times _{S}R$ and $s\in S$.
\end{proof}

\begin{lemma} \label{inducedxmod}
	Let $(\alpha,\id) \colon (P,X,\gamma) \to (S,X,\partial')$ be a crossed module morphism. Then there exists a crossed module $(P,S,\alpha)$  where the action of $S$ on $P$ are defined along $\d'$, namely:
	\begin{align*}
	s \cdot p  = \d'(s) \cdot p , \quad
	s \ast p = \d'(s) \ast p .
	\end{align*}
\end{lemma}

\begin{proof}
	Since $(\alpha,\id)$ is a crossed module morphism, the diagram:
	$$\xymatrix{
		P \ar[dr]^-{\alpha} \ar[dd]_{\gamma}&
		\\ 
		& S \ar[dl]^-{\partial'}
		\\ X  & }
	$$
	commutes; namely $\alpha (x \cdot p) = x \cdot \alpha(p)$ and $\alpha (x \ast p) = x \ast \alpha(p)$, for all $x \in X$ and $p \in P$. Thus:
	
	\medskip
	\begin{itemize}
		\item[XM1)]
		\begin{align*}
		\alpha(s \cdot p) & = \alpha \big( \d'(s) \cdot p \big) = \d'(s) \cdot \alpha(p) = s + \alpha(p) - s \, , \\
		\alpha(s \ast p) & = \alpha \big( \d'(s) \ast p \big)  = \d'(s) \ast \alpha(p)  = s \ast \alpha(p) \, , 
		\end{align*}
		
		\item[XM2)]
		\begin{align*}
		\alpha(p) \cdot p' & =  \d'(\alpha(p)) \cdot p'  = \gamma(p) \cdot p' = p + p' - p \, ,  \\
		\alpha(p) \ast p' & = \d'(\alpha(p)) \ast p' = \gamma(p) \ast p' = p \ast p' \, , 
		\end{align*}
	\end{itemize}
	for all $s \in S$ and $p,p' \in P$.
\end{proof}

\begin{remark}\label{comp}
	If $(A,B,\d)$ and $(B,C,\d')$ are crossed modules such that $C$ acts on $A$ in a compatible way with $B$ (i.e. $(\d' b \cdot a) = b \cdot a$), then $(A,C,\d' \d)$ becomes a crossed module as well, see \cite{Nizar} for details. 
\end{remark}

\begin{lemma}
	Suppose that we have crossed module morphisms:
	$$(\alpha ,\id) \colon (P,X,\gamma) \to (S,X,\d')  \, \text{  and  }  \,(\beta ,\id) \colon (R,X,\delta) \to (S,X,\d').$$ 
	Then there exists a crossed module: 
	\begin{align*}
	P\times _{S}R \to X \, ,
	\end{align*} 
	which leads to the pullback object in $\mathbf{XMod/X}$.
\end{lemma}
	
\begin{proof}
	By using crossed module morphisms $(\alpha ,\id)$ and $(\beta ,\id)$, we get the following morphisms of $\C$: 
	\begin{align*}
	\alpha \colon P \to S \, \text{  and  }  \, \beta \colon R \to S.
	\end{align*}
	
	We already know that the pullback of these morphisms in $\C$ are defined by the fiber product $P\times _{S}R$ that makes the following diagram commutative and satisfies the universal property:
	$$\xymatrix{
		& P\times _{S}R \ar[dl]_-{\pi _{1}} \ar[dr]^-{\pi _{2}} &
		\\ 
		P \ar[dr]_-{\alpha}  & & R \ar[dl]^-{\beta} 
		\\ &S& }
	$$	
	
	By using Lemma \ref{inducedxmod}, $\alpha$ and $\beta$ turn into crossed modules, thus we get a crossed module $\d \colon P\times _{S}R \to S$ in the sense of Lemma \ref{main1}. Moreover, $\d' \colon S \to X$ is already a crossed module and $X$ acts on $P\times _{S}R $ in a natural way. Therefore by using Remark \ref{comp}, we get the crossed module: 
	\begin{align*}
	\d' \d \colon P\times _{S}R \to X ,
	\end{align*} 
	which leads to the pullback object in the category of crossed $X$-modules. All fitting into the diagram:
	$$\xymatrix{
		& P\times _{S}R \ar[dl]_-{\pi _{1}} \ar[dr]^-{\pi _{2}} \ar[dd]^{\d} &
		\\ 
		P \ar[dr]_-{\alpha} \ar@/_1.25pc/[ddr]_{\gamma}  & & R \ar[dl]^-{\beta} \ar@/^1.25pc/[ddl]^{\delta}
		\\ &S \ar[d]^{\d'} & 
		\\ & X &}
	$$	
\end{proof}

\begin{proposition}
	The category of crossed ${X}$-modules has an initial object $0 \to X$ and a terminal object $\id \colon X \to X$. Consequently, one can construct the product object as a pullback of the morphisms:
	$$\xymatrix{
		\mathcal{X} \ar[dr]  & & \mathcal{X'} \ar[dl] 
		\\ & 1 &}
	$$
	where $\mathcal{X,X'}$ are two crossed $X$-modules and $1$ is the terminal object. 
\end{proposition}

This yields the following:

\begin{proposition}
	Given two crossed modules $\alpha \colon P\rightarrow S$ and $\beta \colon R\rightarrow S$ in a modified category of interest $\C$, their product is the crossed module $\partial :P\times _{S} R\rightarrow S$. 
\end{proposition}

Thus, we have proved the following theorem:

\begin{theorem}
	The category $\mathbf{XMod/X}$ is finitely complete.
\end{theorem}

\begin{remark}
	As a consequence of this section, one can obtain the completeness of the categories of crossed X-modules of groups, (commutative) algebras, Lie and Leibniz algebras, dialgebras, etc.
\end{remark}

\section{\bf Pullback Crossed Modules}

\begin{definition}\label{def1}
	For a given crossed module $(P,R,\partial)$ and a morphism
	$\phi \colon S \to R$ in $\mathbb{C}$,  the pullback crossed module is defined as a crossed module morphism:
	\begin{align*}
		(\phi',\phi) \colon \phi^{\star}(P,R,\d) \to (P,R,\d) \, ,
	\end{align*}
	 where the crossed module: 
	\begin{align*}
	\phi ^{\star }(P,R,\partial
	)=\left( \phi ^{\star }(P),S,\partial ^{\star }\right) 
	\end{align*}
    satisfies the following universal property.
	
	\medskip
		
		For any crossed module morphism: $$ \left( f,\phi \right)  \colon \left( X,S,\mu \right) \rightarrow (P,R,\partial ) \, ,$$
		there exists a unique crossed module morphism: $$\left( f^{\star
		},\id_{S}\right) :\left( X,S,\mu \right) \rightarrow \left( \phi ^{\star
	}(P),S,\partial ^{\star }\right) $$ such that the following diagram commutes:
	\begin{align}\label{cat1}
	\xymatrix@R=40pt@C=40pt{
		& & & (X,S,\mu ) \ar[d]^{(f,\phi)} \ar@{-->}[dlll]_{(f^{\star },\id_{S})}  \\
		\ \left( \phi ^{\star }(P),S,\partial^{\star }\right) \ar[rrr]_{(\phi',\phi)} &  & & (P,R,\partial)  } 
	\end{align}

 In other words, it can be seen as a pullback \cite{Adamek} diagram:
\begin{align}\label{cat2}
\xymatrix@R=20pt@C=20pt{
	X \ar[dd]_{\mu} \ar[rr]^{f} \ar@{.>}[dr]|-{f^{\star}}
	&  & P \ar[dd]^{\partial}  \\
	&\phi^{\star}(P)\ar[ur]_{\phi ^{\prime }}\ar[dl]^{\partial^{\star}}&
	\\ S  \ar[rr]_{\phi}
	& & R             }
\end{align}

\end{definition}

In order to give a particular construction for the pullback crossed module, let $(P, R, \d)$ be a crossed module and $\phi
:S\rightarrow R$ be a morphism in $\mathbb{C}$. Define:
\begin{equation*}
\phi ^{\star }(P)=P\times _{R}S=\left\{ \left( p,s\right) \mid \partial
\left( p\right) =\phi \left( s\right) \right\} ,
\end{equation*}%
and define the morphism $\partial ^{\star
}:\phi ^{\star }(P)\rightarrow S$ by: 
\begin{align*}
\partial ^{\star }\left( p,s\right) =s.
\end{align*} 

\medskip

There exists an action of $S$ on $\phi ^{\star }(P)$ defined by:
\begin{equation*}
\begin{array}{ccl}
S \times \phi ^{\star }(P)  & \rightarrow  & \phi ^{\star }(P) \\
\left(t, \left( p,s\right)\right)  & \mapsto  &t \cdot \left( p,s\right) =\left( \phi \left( t \right) \cdot p
, t + s - t \right) ,
\end{array}
\end{equation*}
and:
\begin{equation*}
\begin{array}{ccl}
S \times \phi ^{\star }(P)  & \rightarrow  & \phi ^{\star }(P) \\
\left(t, \left( p,s\right)\right)  & \mapsto  &t \ast \left( p,s\right) =\left( \phi \left( t \right) \ast p
, t \ast s \right) .%
\end{array}%
\end{equation*}

\medskip

Then $(\phi^{\star }(P) , S , \partial ^{\star })$ defines a crossed module since:
\medskip
\begin{itemize}
	
	\item[XM1)]
	\begin{align*}
	\partial ^{\star }\left(t \cdot \left( p,s\right) \right) & = \partial ^{\star }\left( \phi \left( t \right)\cdot p
	, t + s - t \right)  =  t + s - t  = t + \partial ^{\star }\left( p,s\right) - t \, , \\
	\partial ^{\star }\left(t \star \left( p,s\right) \right) & = \partial ^{\ast }\left( \phi \left( t \right)\ast p
	, t \ast s \right)  = t \ast s  = t \ast \partial ^{\star }\left( p,s\right) \, , 
    \end{align*}%

	\item[XM2)] 
	\begin{align*}
	\partial ^{\star }\left( p^{\prime },s^{\prime
	}\right) \cdot \left( p,s\right)  & =  s' \cdot \left( p,s\right)    \\
	& =  \left( \phi \left( s^{\prime }\right) \cdot p ,s' + s - s' \right)  \\
	& =  \left( \partial \left( p^{\prime }\right)\cdot p ,s' + s - s' \right)  \\
	& =  \left( p' + p - p', s' + s - s' \right)  \\
	& = \left( p',s'\right) + \left( p, s\right) - (p',s')  \, ,
	\end{align*}%
	\begin{align*}
	\partial ^{\star }\left( p^{\prime },s^{\prime
	}\right) \ast \left( p,s\right)  & =  s' \ast \left( p,s\right)    \\
	& =  \left( \phi \left( s^{\prime }\right) \ast p ,s'\ast s \right)   \\
	& =  \left( \partial \left( p^{\prime }\right)\ast p ,s' \ast s\right)   \\
	& =  \left( p'\ast p, s'\ast s \right) \\
	& =  \left( p',s'\right) \ast \left( p, s\right)  ,
	\end{align*}	
\end{itemize}
for all $(p,s) \, (p',s') \in \phi^{\star} (P)$ and $t \in S$.

\bigskip

This construction satisfies the universal property. Consider the crossed module morphism:
\begin{equation*}
\left( \phi ^{\prime },\phi \right) :\left( \phi ^{\star }(P),S,\partial
^{\star }\right) \rightarrow (P,R,\partial ) \, ,
\end{equation*}%
where $\phi ^{\prime }:\phi ^{\star }(P)\rightarrow P$ is defined by $\phi
^{\prime }\left( p,s\right) =p.$

\bigskip

Suppose that $\left( X,S,\mu \right) $ is a crossed module and the tuple:
\begin{align}\label{com2}
\left( f,\phi \right) :\left( X,S,\mu \right) \rightarrow (P,R,\partial )
\end{align}
is a crossed module morphism. 

\medskip

Define: $f^{\star }:X\rightarrow \phi ^{\star }(P)$ by $f^{\star }(x)=\left(
f\left( x\right) ,\mu \left( x\right) \right).$ Then:
\begin{equation*}
\left( f^{\star },\id_{S}\right) :\left( X,S,\mu \right) \rightarrow \left(
\phi ^{\star }(P),S,\partial ^{\star }\right)
\end{equation*}
becomes a crossed module morphism. In fact the diagram:
$$ \xymatrix @R=20pt@C=20pt{
	X\ar[r]^{\mu} \ar[d]_{f^{\star }} &  S  \ar[d]^{\id_{S}} \\
	\ \phi ^{\star }(P) \ar[r]_-{\partial^{\star }} & S  }  $$
is commutative since:
\begin{align}\label{com1}
\begin{split}
\partial ^{\star }f^{\star }(x) = \partial ^{\star }\left( f\left( x\right)
,\mu \left( x\right) \right)  = \mu \left( x\right)  = \id_{S} \mu \left( x\right) ,
\end{split} 
\end{align}
and also:
\begin{align*}
f^{\star }\left( s\cdot x\right) & = \left( f\left( s\cdot x \right) ,\mu
\left( s\cdot x\right) \right)   \\
& = \left(\phi \left( s\right) \cdot f\left( x\right) ,s \cdot \mu \left(
x\right) \right)  \\
& = s \cdot \left( f\left( x\right) ,\mu \left( x\right) \right)   \\
& =  \id_{S}\left( s\right) \cdot f^{\star }(x) ,
\end{align*}
\begin{align*}
f^{\star }\left( s\ast x\right) & = \left( f\left( s\ast x \right) ,\mu
\left( s\ast x\right) \right) \\
& =  \left(\phi \left( s\right) \ast f\left( x\right) ,s \ast \mu \left(
x\right) \right)   \\
& = s \ast \left( f\left( x\right) ,\mu \left( x\right) \right) \\
& = \id_{S}\left( s\right) \ast f^{\star }(x) ,
\end{align*}
for all $s\in S$ and $x\in X$. Moreover we have:
\begin{align}\label{com3}
\begin{split}
\phi ^{\prime }f^{\star }(x) & = \phi ^{\prime }\left( f\left( x\right)
,\mu \left( x\right) \right)   =  f\left( x\right) 
\end{split}
\end{align}
that makes diagram \eqref{cat1} commutative. In other words, pullback diagram \eqref{cat2} commutes
since:
\begin{itemize}
	\item  $\partial ^{\star }f^{\star }(x) =  \mu \left( x\right)$  from \eqref{com1}, 
	\item  $\phi \mu = \partial f$ since \eqref{com2} is a crossed module morphism,
	\item $\phi' f^{\star} = f$ from \eqref{com3}.
\end{itemize}

\medskip

Finally, we need to prove that $(f^{\star},\id)$ is unique in \eqref{cat1}. Suppose that:
\begin{align*}
\left( f^{\star\star
},\id_{S}\right) :\left( X,S,\mu \right) \rightarrow \left( \phi ^{\star
}(P),S,\partial ^{\star }\right) 
\end{align*}
is a crossed module morphism with the same property as $(f^{\star},\id)$. We get:
\begin{align*}
\partial ^{\star }f^{\star\star }(x) =f(x) , \quad \quad
\partial ^{\star }f^{\star\star }(x)=\mu(x) ,
\end{align*}
for all $x \in X$ which implies:
\begin{align*}
f^{\star\star}(x)=(p,s)=(f(x),\mu(x))=f^{\star}(x),
\end{align*}
and proves that $(f^{\star},\id)$ is unique.

\bigskip

Therefore we have the following:
\begin{corollary}
	We get a functor in $\C$:
	\begin{align*}
	\phi^{\star} \colon \mathbf{XMod/{R} \to XMod/{S} } .
	\end{align*}
Moreover, let $(P,R,\partial )$ be a crossed module  and $\phi \colon S \to R
	$ be a morphism in $ \C $. We have the pullback diagram:
	
	$$ \xymatrix @R=40pt@C=40pt{
		\phi ^{\star }(P) \ar[r]^{\phi ^{\prime }} \ar[d]_{\partial^{\star }} &  P  \ar[d]^{\partial} \\
		\ S \ar[r]_\phi & R  } $$
	
\end{corollary}

\begin{example}
	Given an object $R$ and a normal subobject $N$ of $R$, then $(N,R,\d)$
	is a crossed module where $\d$ is the inclusion map. Suppose that $\phi :S\rightarrow R$ is a morphism. Then the pullback crossed module is defined by:
	\begin{align*}
	\phi ^{\star }\left( N\right) & = \left\{ \left( n,s\right) \mid \partial
	\left( n\right) =\phi \left( s\right) \text{, }n\in N\text{, }s\in S\right\}
	\\
	& \cong  \left\{ s\in S\mid \phi \left( s\right) =n,\text{ }n\in N\right\}
	\\
	& = \phi ^{-1}\left( N\right) ,
	\end{align*}%
	and the pullback diagram is:
	$$ \xymatrix @R=40pt@C=40pt{
		\phi ^{-1}\left( N\right) \ar[r]^-{\phi ^{\prime }} \ar[d]_{\partial^{\star }} &  N  \ar[d]^{\partial} & \\
		\ S \ar[r]_\phi & R &  } $$
	where the preimage $\phi ^{-1}\left( N\right) $ is a normal subobject of $S$. 
	
	\medskip
	
	In particular, if $N=\left\{ 0 \right\} $, then:
	\begin{equation*}
	\phi ^{\star }\left( \left\{ 0\right\} \right) \cong \left\{ s\in S\mid \phi
	\left( s\right) =0\right\} =\ker \phi .
	\end{equation*}%
	So kernels are particular cases of pullback crossed modules.
\end{example}

\section{\bf Pullback Cat$^1$-Objects}

\begin{definition}
	The definition of pullback \cat object along a morphism is similar to that for crossed modules given in Definition \ref{def1}. For a given \cat object $(R,S)$ and a morphism $\phi \colon Q \to S$ in $\mathbb{C}$, we require a \cat object $\phi ^{\star }(R,S)=\left( \phi ^{\star}(R),Q\right)$ to fill the pullback diagrams:

		\begin{align}\label{cat3}
		 \xymatrix@R=40pt@C=40pt{
			& & & (P,Q ) \ar[d]^{(\varphi,\phi)} \ar@{-->}[dlll]_{(\psi,\id_{S})}  \\
			\ \left( \phi ^{\star}(R),Q\right) \ar[rrr]_{(\pi,\phi)} &  & & (R,S)  }
		\end{align}
and
	\begin{align}\label{cat4}
	\xymatrix{
		P \ddto\ddto<.5ex>^{s^{\prime }}_{t^{\prime }} \ar[rr]^{\varphi} \ar@{.>}[dr]|-{\psi}
		&  & R \ddto\ddto<.5ex>^{s}_{t}  \\
		&\phi^{\star}(R)\ar[ur]_{\pi}\dlto\dlto<.5ex>^{s^{\star}}_{t^{\star}}   &
		\\ Q  \ar[rr]_{\phi}
		& &S            }
	\end{align}

	Note that we do not include the embedding morphisms in the above diagrams for the sake of simplicity.
\end{definition}

In order to give a particular construction for the pullback \cat object, let $\left( e;s,t:R\rightarrow S\right) $ be a \cat object and $\phi
:Q\rightarrow S$ be a morphism. Define:
\begin{equation*}
\phi ^{\star }\left( e;s,t:R\rightarrow S\right)  = \left( e^{\star };s^{\star },t^{\star }:\phi ^{\star }(R)\rightarrow
Q\right) ,
\end{equation*}
where:
\begin{equation*}
\phi ^{\star }(R)=\left\{ \left( q_{1},r,q_{2}\right) \in Q\times R\times
Q\mid \phi \left( q_{1}\right) =s\left( r\right) ,\text{ }\phi \left(
q_{2}\right) =t\left( r\right) \right\}
\end{equation*}
is a subobject of $Q \times R \times Q$. 

\medskip

Define the morphisms:
\begin{align*}
s^{\star }\left( q_{1},r,q_{2}\right) =q_{1}, \quad t^{\star }\left(q_{1},r,q_{2}\right) =q_{2}, \quad e^{\star }\left( q\right)  =\left( q,e\phi
\left( q\right) ,q\right).
\end{align*}
It is easily verified that $s^{\star}e^{\star }=t^{\star }e^{\star}= \id_{Q}$.

\medskip

Moreover, let $\left( q_{1}^{\prime },r_{1},q_{1}\right) \in \ker s^{\star }$ and
$\left( q_{2},r_{2},q_{2}^{\prime }\right) \in \ker t^{\star }.$ Then:

\begin{equation*}
s^{\star }\left( q_{1}^{\prime },r_{1},q_{1}\right) = 0_Q \quad \text{and} \quad t^{\star }\left( q_{2},r_{2},q_{2}^{\prime }\right) =0_Q,
\end{equation*}
which implies $q_{1}^{\prime }=q'_2 = 0_Q$, hence we get $r_1 \in \ker s$ and $r_2 \in \ker t$.

\medskip

Therefore:
\begin{align*}
\left( q_{1}^{\prime },r_{1},q_{1}\right) \ast \left(
q_{2},r_{2},q_{2}^{\prime }\right) & =\left( 0_Q \ast
q_{2},r_{1}\ast r_{2},q_{1}\ast 0_Q \right) \\
& =\left( 0_Q , r_{1} \ast r_2 , 0_Q \right) \\
& = ( 0_Q , 0_R , 0_Q) ,
\end{align*}
and:
\begin{align*}
\left( q_{1}^{\prime },r_{1},q_{1}\right) + \left(
q_{2},r_{2},q_{2}^{\prime }\right) & =\left( 0_Q +
q_{2},r_{1} + r_{2},q_{1} + 0_Q \right) \\
& =\left( q_2 , r_{2} + r_1 , q_1 \right) \\
& = \left(
q_{2},r_{2},q_{2}^{\prime }\right) + \left( q_{1}^{\prime },r_{1},q_{1}\right) ,
\end{align*}
which implies:
\begin{align*}
\left( q_{1}^{\prime },r_{1},q_{1}\right) + \left(
q_{2},r_{2},q_{2}^{\prime }\right) - \left( q_{1}^{\prime },r_{1},q_{1}\right) - \left(
q_{2},r_{2},q_{2}^{\prime }\right) = 0_{\phi^{\ast}(R)}.
\end{align*}
Consequently, we get the \cat object structure: 
$$\left( e^{\star};s^{\star},t^{\star}:\phi ^{\star}(R)\rightarrow Q\right). $$ 

Define the morphism:
\begin{equation*}
\begin{array}{cccl}
\pi : & \phi ^{\star}(R) & \rightarrow & R \\
& \left( q_{1},r,q_{2}\right) & \mapsto & \pi \left( q_{1},r,q_{2}\right) =r.%
\end{array}
\end{equation*}
Since:
\begin{align*}
\phi s^{\star}\left( \left( q_{1},r,q_{2}\right) \right) & =\phi \left(
q_{1}\right)  =s\left( r\right)  =s\pi \left( q_{1},r,q_{2}\right) , \\
\phi t^{\star}\left( \left( q_{1},r,q_{2}\right) \right) & =\phi \left(
q_{2}\right) =t\left( r\right) =t\pi \left( q_{1},r,q_{2}\right) , \\
\pi e^{\star }\left( q\right) & =\pi \left( q,e\phi \left( q\right) ,q\right) =e\phi \left( q\right),
\end{align*}
for all $\left( q_{1},r,q_{2}\right) \in \phi ^{\star }(R)$, $q\in Q$, the following diagram commutes:
$$
\diagram
\phi ^{\star }(R)\ddto \ddto<.7ex>^{s^{\star}}_{t^{\star}} %
\rrto^{\pi} &&  R\ddto<-.8ex>_{t} \ddto<-.1ex>^{s}\\
\\
Q \rrto_{\phi} \ar@/^1.25pc/[uu]^{e^{\star}} && S \ar@/_1.25pc/[uu]_{e}
\enddiagram
$$
Hence $(\pi,\phi)$ becomes a \cat object morphism.

\medskip

Now we need to prove the universal property. Let:
\begin{align*}
\left( \varphi ,\phi \right) \colon  \left( e^{\prime };s^{\prime
},t^{\prime }:P\rightarrow Q\right) \rightarrow \left( e;s,t:R\rightarrow
S\right) 
\end{align*}
be any \cat  morphism such that the following diagram commutes:
$$
\diagram
P\ddto\ddto<.5ex>^{s^{\prime }}_{t^{\prime }} %
\rrto^{\varphi} &&  R\ddto \ddto<.5ex>^{s}_{t}\\
\\
Q \rrto_{\phi} \uutol^{{e^{\prime }}} && S \uutor_{e}
\enddiagram
$$

Define $\psi :P\rightarrow \phi
^{\star}(R)$ by $
\psi \left( p\right) =\left( s^{\prime }\left( p\right)
,\varphi \left( p\right) ,t^{\prime }\left( p\right) \right)$. Then:
\begin{equation*}
\left( \psi
,\id_{Q}\right) \colon \left( e^{\prime };s^{\prime },t^{\prime
}:P\rightarrow Q\right) \rightarrow \left( e^{\star };s^{\star
},t^{\star}:\phi ^{\star }(R)\rightarrow Q\right) 
\end{equation*}%
becomes a \cat object since:
\begin{align*}
s^{\star }\psi \left( p\right) & =s^{\star }\left( s^{\prime }\left( p\right)
,\varphi \left( p\right) ,t^{\prime }\left( p\right) \right) =s^{\prime
}\left( p\right)  =\id_{Q}s^{\prime }\left( p\right) , \\
t^{\star }\psi \left( p\right) & =t^{\star}\left( s^{\prime }\left( p\right)
,\varphi \left( p\right) ,t^{\prime }\left( p\right) \right)  =t^{\prime
}\left( p\right) =\id_{Q}t^{\prime }\left( p\right) ,
\end{align*}%
and also:
\begin{align*}
\psi e^{\prime }\left( q\right) & =\left( s^{\prime }e^{\prime }\left(
q\right) ,\varphi e^{\prime }\left( q\right) ,t^{\prime }e^{\prime }\left(
q\right) \right)  =\left( q,e\phi \left( q\right) ,q\right)  =e^{\star } \id_{Q}\left( q\right) ,
\end{align*}
for all $p \in P$ and $q \in Q$. Moreover we have:
\begin{align*}
\pi \psi \left( p\right) & =\pi \left( s^{\prime }\left( p\right) ,\varphi
\left( p\right) ,t^{\prime }\left( p\right) \right)   =\varphi \left( p\right),
\end{align*}
which makes \eqref{cat3} commutative and leads to the pullback diagram \eqref{cat4}. The uniqueness of $\left( \pi ,\phi \right) $ can be proven analogously to the crossed module case given in the previous section.

\section{\bf Conclusion}

Consequently, we get the commutativity of the following diagram (up to isomorphism) for a fixed morphism $\phi$ of $\C$.
$$ \xymatrix @R=40pt@C=40pt{
	\mathbf{XMod} \ar[r]^{\phi^{\star}} \ar[d]_{\mathbf{C^{1}}} &  \mathbf{XMod}  \ar[d]^{\mathbf{C^{1}}} \\
	\mathbf{Cat^{1}} \ar[r]_{\phi^{\star}} & \mathbf{Cat^{1}}  } $$

\medskip

Another main outcome of the paper is the following:

\medskip

One can obtain pullback crossed modules and pullback \cat objects in many well-known algebraic categories listed in Example \ref{example1}, such as category of groups, (commutative) algebras, dialgebras, Lie algebras and Leibniz algebras, etc. For instance, if we consider the cases of category of groups and commutative algebras, we lead to the constructions given in \cite{Alp1, Alp2, Brown}.

\section*{\bf Acknowledgments}
The authors are thankful to Enver \"Onder Uslu, Murat Alp and the anonymous referee for their invaluable comments and suggestions.

\end{document}